\documentclass[preprint,12pt]{elsarticle}

\usepackage{amssymb}
\journal{Journal of Computational and Applied Mathematics }

\usepackage{graphicx}%
\usepackage{multirow}%
\usepackage{amsmath,amssymb,amsfonts}%
\usepackage{amsthm}%
\usepackage{mathrsfs}%
\usepackage[title]{appendix}%
\usepackage{xcolor}%
\usepackage{textcomp}%
\usepackage{manyfoot}%
\usepackage{booktabs}%
\usepackage{algorithm}%
\usepackage{algorithmicx}%
\usepackage{algpseudocode}%
\usepackage{listings}%

\newtheorem{theorem}{Theorem}[section]%  meant for continuous numbers
\newtheorem{example}{Example}[section]%
\newtheorem{remark}{Remark}[section]%

\newtheorem{coro}{Corollary}[section]% 

\newtheorem{prop}{Proposition}[section]%

\newtheorem{assumption}{Assumption}[section]

\usepackage{color}

\usepackage{algorithm}
\usepackage{algpseudocode}
\usepackage{amsmath}
\usepackage{hyperref}

\usepackage{booktabs}

\usepackage{adjustbox}
\usepackage{tikz}
\usetikzlibrary{positioning}
\usetikzlibrary{chains}
\usepackage{tikz-cd}
\usepackage{tikz}
\usepackage{pgfplots}
\pgfplotsset{compat=1.16}
\usetikzlibrary{calc,intersections}

\usetikzlibrary{shapes.geometric, arrows, positioning}

\tikzset{
   mybox/.style  = {draw, rectangle, minimum width=4cm, minimum height=0.8cm, text centered, text width=4.4cm,   
  font=\normalsize},
  box/.style  = {draw, rectangle, minimum width=4.0cm, minimum height=0.8cm, text centered, text width=4.5cm,  font=\normalsize},
   myarrow/.style = {line width=0.2pt, draw=black, -triangle 60, postaction={draw, line width=0.2pt, shorten >=10pt,-}}
}

\tikzstyle{arrow} = [->, >=stealth]
\usepackage{latexsym}
\usepackage{mathrsfs}
\usepackage{amssymb}

\usepackage{amsfonts}
\usepackage{subfigure}

\usepackage{comment}
\usepackage{multirow}

\newcommand{\x}{\boldsymbol{x}}

\newcommand{\z}{\boldsymbol{z}}
\newcommand{\I}{\boldsymbol{I}}

\usepackage{lipsum}
\usepackage{fancyhdr}
\usepackage{array,contour}

\usepackage{pifont}

\allowdisplaybreaks[4]

\newcommand{\rev}[1]{{#1}}
\newcommand{\revb}[1]{{#1}}
\newcommand{\revc}[1]{{#1}}

\numberwithin{equation}{section} % Number equations by section

\begin{document}

\begin{frontmatter}


 %\title{\rev{Sufficient Conditions for Error Distance Reduction in the \(\ell^2\)-norm Trust Region between Minimizers of Local Nonconvex Multivariate Quadratic Approximates}} 
 \title{\rev{Sufficient Conditions for Error Distance Reduction in the \(\ell^2\)-norm Trust Region between Minimizers of Local Nonconvex Multivariate Quadratic Approximates}}%\tnoteref{label1}
%\tnotetext[label1]{}
 \author{Pengcheng Xie\corref{cor1}}%\fnref{label2}
 \ead{pxie@lbl.gov}
% \affiliation{organization={},
\affiliation{organization={Applied Mathematics and Computational Research Division, Lawrence Berkeley National Laboratory},%Department and Organization
            addressline={1 Cyclotron Road}, 
            city={Berkeley},
            postcode={94720}, 
            state={CA},
            country={USA}}

\begin{abstract}
This paper analyzes the sufficient conditions for distance reduction between minimizers of \rev{local {\rev{nonconvex}} quadratic approximate functions with}  diagonal Hessian in the $\ell^2$-norm trust regions after two iterations. \rev{Some examples illustrate the theoretical
results of this study.} %\revc{This paper also presents some examples corresponding to the theoretical results.}

\end{abstract}


% \begin{highlights}
% \end{highlights}

\end{frontmatter}


\section{Introduction}\label{Introduction}

Goldfeld, Quandt, and Trotter \cite{goldfeld1966maximization} made a crucial advancement in the trust-region algorithm \cite{conn2000trust,yuan2015recent} in 1966 by introducing an explicit procedure to update the maximum step size. Although it is unclear from their paper whether they considered this parameter as Hessian damping, which imposes a restriction on step size, or as a step size restriction calculated by damping the Hessian, their update procedure closely resembles the one currently used in trust-region algorithms. The concept of ``achieved versus predicted change'' was also introduced, which compares the actual reduction in the objective function with the reduction predicted by the quadratic approximation. The proposed method in 1966 utilizes the same quadratic approximation as Newton's Method but with a damping parameter in the Hessian matrix that limits the step size. This damping parameter is adjusted based on the accuracy of the quadratic approximation to increase step size in areas of good approximation and decrease it in areas of poor approximation.

\begin{sloppypar}
The trust-region method obtains the next iteration point by solving the subproblem 
\begin{equation*}
\begin{aligned}
\min_{\x \in {\revc{\mathbb{R}}}^n} \ & m(\x)\\
\text { subject to }&\|\x-\x_c\|_2 \leq \Delta_k,
\end{aligned}
\end{equation*}
where \(\x_c\in\mathbb{R}^n\) is the center of the \(\ell^2\) trust region \revc{\(\mathcal{B}(\x_c,\Delta_k)=\{\z\in\mathbb{R}^n, \left\|\z-\x_c\right\|_2\le \Delta_k\}\)}, \(\Delta_k>0\) is a trust-region radius for any \(k\), and \(\rev{m: \mathbb{R}^n\rightarrow \mathbb{R}}\) here is a \rev{local quadratic approximate function} of the objective function we need to minimize. \revb{Notice that the aim of solving the trust-region subproblem in such methods is to find a better iteration point with a lower objective function value within the corresponding trust region by using the approximate function. Such a region is determined iteratively according to the function value at the obtained trial point to ensure that the approximate function is numerically accurate enough at each iteration \cite{conn2000trust}. Ultimately, we will obtain a numerical approximate  minimizer of the objective function in the case without any constraints. Here, the solution of the trust-region subproblem is an approximate point to the minimizer of the objective function within the same trust region.} 
%Notice that the aim of solving the trust-region subproblem in such kinds of  methods is to find a better iteration point with lower function value of \(f\) in the corresponding trust region, of which the region is determined iteratively to guarantee the approximate function is accurate enough numerically. We will finally obtain the numerical minimizer of the objective function \(f\) without any constraint.}
Therefore, a \rev{local quadratic approximate function}  is important for giving the next iteration point. This paper discusses the {\rev{nonconvex}} case. \revb{It considers the distance of minimizers of two nonconvex quadratic functions in the corresponding trust regions. One of the motivation is that local quadratic approximate functions are used to provide an approximate minimizer when solving problems using trust-region methods.}  
% to observe and reveal more with the aim that the \rev{local quadratic approximate function}s are used to provide an approximate minimizer, which is naturally leveraged from the trust-region method using \rev{local quadratic approximate} functions.}
\end{sloppypar}

\revb{This paper will give the condition where there is a reduction of the distance between the two minimizers of such two local quadratic approximate functions \(f\) and \(Q\).}  
%\rev{local quadratic approximate function}s can provide the , which appears in Theorem \ref{thm-1-1} and Theorem \ref{sufficient_condition_general_n}
The results are helpful for iteratively modifying the \rev{local quadratic approximate function}s or dealing with the choice of the \rev{local quadratic approximate function}s for derivative-based or derivative-free trust-region methods \cite{conn2009introduction,Powell2003,Xie_Yuan_2022,xie2023derivativefree,xieyuansusdtr,xie2003jpcs}. Besides, we use the examples to show that our results are applicable. For example, we can directly use such conditions to tell whether the two different \rev{local quadratic approximate function}s can provide a reduction of the distance between the two minimizers after \revb{an iteration step}. 

%, where the reduction of the distance between two minimizers is essential in cases where there are errors or perturbations for the \rev{local quadratic approximate function}s

{Notice that the quadratic functions \(f\) and \(Q\) refer to the  \rev{local quadratic approximate} functions in trust-region algorithms. \revc{One should mention that \(f\) is not the original objective function, although it can be in cases where we want to minimize a quadratic function.  The quadratic functions \(f\) and \(Q\) are both \rev{local quadratic approximate function}s appearing in the trust-region subproblem}.}
These values can be chosen such that when two trust-region steps are executed starting from $\x_{0}$, the distance between the trust-region solutions referring to the functions $f$ and $Q$ obtained at the second step is smaller than or equal to the distance between the two solutions obtained at the first step.

In a word, the distance between the minimizers of two {\rev{nonconvex}} quadratic functions in the corresponding trust regions will reduce in some cases. This paper derives the sufficient conditions for such cases.  

{\bf Notation.} %\label{big-assumption-2-k-1}
In the following, we suppose that  \rev{$\x_{1}, \tilde{\x}_{1}\in \revc{\mathbb{R}}^{n}$} are respectively the minimizers of the {\rev{nonconvex}} \rev{multivariate} quadratic functions $f$ and $Q$ (with \(n\) variables) in the trust region \revc{$\mathcal{B}(\x_{0},\Delta_{1})$} and \revc{$\mathcal{B}(\x_{0},\tilde{\Delta}_{1})$}, and  $\x_{2}$ and $\tilde{\x}_{2}$ are respectively the minimizers of $f$ and $Q$ in the trust region \revc{$\mathcal{B}(\x_{1},\Delta_2)$} and \revc{$\mathcal{B}(\tilde{\x}_{1},\tilde{\Delta}_2)$}, where \(\x_0\in\revc{\mathbb{R}}^{n}\) is the initial point (or the center of the first trust region), and \rev{\(\Delta_1, \tilde{\Delta}_1, \Delta_2, \tilde{\Delta}_2\in \revc{\mathbb{R}}^{+}\)} are the trust-region radii. In other words, it holds that there exist real parameters $\omega_1,\tilde{\omega}_1,{\omega}_2,\tilde{\omega}_2>0$ such that 
\begin{equation}
\label{KKT-lambda-1-2-k-1}
\left\{
\begin{aligned}
&\x_{1}-\x_{0}=-\left(\nabla^{2} f+\omega_{1} \boldsymbol{I}\right)^{-1} \nabla f\left(\x_{0}\right),\\
&\tilde{\x}_{1}-\x_{0}=-(\nabla^{2}Q+\tilde{\omega}_1 \boldsymbol{I})^{-1} \nabla Q\left(\x_{0}\right), 
\end{aligned}
\right.
\end{equation}
and
\begin{equation*}
\label{KKT-lambda-3-4-k-1}
\left\{
\begin{aligned}
&\x_{2}-\x_{1}=-\left(\nabla^{2} f+{\omega}_2 \boldsymbol{I}\right)^{-1} \nabla f\left(\x_{1}\right),\\
&\tilde{\x}_{2}-\tilde{\x}_{1}=-(\nabla^{2}Q+\tilde{\omega}_2 \boldsymbol{I})^{-1} \nabla Q\left(\tilde{\x}_{1}\right),
\end{aligned}
\right.
\end{equation*}
where 
% \begin{equation*}
% \begin{aligned}
\( 
\Delta_{1}=\Vert \x_{1}-\x_{0} \Vert_2,\ 
\tilde{\Delta}_{1}=\Vert \tilde{\x}_{1}-\x_{0} \Vert_2,\  
\Delta_{2}=\Vert \x_{2}-\x_{1} \Vert_2,\ 
\tilde{\Delta}_{2}=\Vert \tilde{\x}_{2}-\x_{1} \Vert_2,
\) 
% \end{aligned}
% \end{equation*}
and 
\(\nabla^2 f+\omega_1\boldsymbol{I}\succeq \boldsymbol{0},  \nabla^2 Q+\tilde{\omega}_1\boldsymbol{I}\succeq \boldsymbol{0}, \nabla^2 f+{\omega}_2\boldsymbol{I}\succeq \boldsymbol{0},  \nabla^2 Q+\tilde{\omega}_2\boldsymbol{I}\succeq \boldsymbol{0}\).

% \end{notation}

\begin{assumption}

\label{big-assumption-2-k-1}

Suppose that \(f\) and $Q$ are {\rev{nonconvex}}  quadratic functions, \(\nabla^2 f+{\omega}_2\succ \boldsymbol{0},  \nabla^2 Q+\tilde{\omega}_2\succ \boldsymbol{0}\), and $\tilde{\x}_1\ne {\x}_{1}$. 

\end{assumption}

\begin{remark}%\color{red}
\rev{We use the same notations for different dimensions of cases for clearness and simplicity purposes.} 
%The same notations we use in different results may refer to different dimensions, and we use them in this way for clearness and simplicity.
\end{remark}

\rev{
The question we are discussing is to know if there exists under Assumption \ref{big-assumption-2-k-1} any sufficient condition of local approximates \(f\) and \(Q\) for 
\begin{equation}
\label{<q}
\left \Vert \tilde{\x}_{2}-\x_{2}\right\Vert_2 \le \revc{\rho}\Vert \tilde{\x}_1-{\x}_{1} \Vert_2, %\ %?
\end{equation} 
where \(0\le \rho\le 1\). } 
%\end{question}

%%%%
%%%%The question we discuss is in the following. 
%%%%
%%%%%\begin{question}
%%%%\noindent {\bf \color{red}Question.} 
%%%%Under Assumption \ref{big-assumption-2-k-1}, for \(0\le\revc{\rho}\le1\), what is the sufficient condition of the quadratic functions \(f\) and \(Q\) for 
%%%%\begin{equation}
%%%%\label{<q}
%%%%\left \Vert \tilde{\x}_{2}-\x_{2}\right\Vert_2 \le \revc{\rho}\Vert \tilde{\x}_1-{\x}_{1} \Vert_2\ ?
%%%%\end{equation} 
%%%%%\end{question}
%%%%

%\section[Distance analysis of the minimizers]{Distance analysis of the minimizers of quadratic functions in the trust region}

\section[Error distance analysis of approximate minimizers]{\rev{Error distance analysis of approximate minimizers}}

\subsection{\rev{Error distance between minimizers}}

{To observe the advantages of considering the optimality when constructing the \rev{local quadratic approximate function} for trust-region methods, we discuss the \rev{error distance} between minimizers of quadratic functions.}

\begin{prop}
The \rev{gap between minimizers} satisfies that 
\begin{equation*}
\tilde{\x}_{2}-\x_{2} =\tilde{\omega}_1\left(\nabla^{2} Q+\tilde{\omega}_2 \boldsymbol{I}\right)^{-1}  (\tilde{\x}_{1}-\x_{0})-\omega_1\left(\nabla^{2} f+{\omega}_2 \boldsymbol{I}\right)^{-1} (\x_{1}-\x_{0})+(\tilde{\x}_1-{\x}_{1}).
\end{equation*}

\end{prop}

\begin{proof} 
It holds that 
\begin{equation*}
\begin{aligned}
\x_{2}-{\x}_{1} &=-\left(\nabla^{2} f+{\omega}_2 \boldsymbol{I}\right)^{-1} \nabla f\left(\x_{1}\right)\\
&{=-\left(\nabla^{2} f+{\omega}_2 \boldsymbol{I}\right)^{-1}\left(\nabla f(\x_{0})+\nabla^{2} f \cdot (\x_{1}-\x_{0})\right) }\\
&{=-\left(\nabla^{2} f+{\omega}_2 \boldsymbol{I}\right)^{-1}\left(-\left(\nabla^{2} f+\omega_{1} \boldsymbol{I}\right)(\x_{1}-\x_{0})+\nabla^{2} f \cdot (\x_{1}-\x_{0})\right) }\\
&=\omega_1\left(\nabla^{2} f+{\omega}_2 \boldsymbol{I}\right)^{-1}  (\x_{1}-\x_{0}),
\end{aligned}
\end{equation*} 
\revc{and}
\begin{equation*}
\begin{aligned}
\tilde{\x}_{2}-\tilde{\x}_1=\tilde{\omega}_1\left(\nabla^{2} Q+\tilde{\omega}_2 \boldsymbol{I}\right)^{-1}  (\tilde{\x}_{1}-\x_{0}),
\end{aligned}
\end{equation*}
according to the relationship of $\nabla f(\x_{0})$, $\nabla f(\x_{1})$, $\nabla f(\tilde{\x}_{1})$ and (\ref{KKT-lambda-1-2-k-1}). 
Then the conclusion can be  proved directly according to the inequality of the norm. 
\end{proof}

\begin{theorem}[sufficient and necessary condition for 1-dimensional case]
\label{thm-1-1}
{Suppose that Assumption \ref{big-assumption-2-k-1} holds, the dimension \(n=1\) and \(\kappa :=\frac{{\x}_{1}-\x_{0}}{\tilde{\x}_1-{\x}_{1}}\in \revc{\mathbb{R}}\). Then (\ref{<q}) holds for \(0\le \revc{\rho}\le 1\) if and only if 
\begin{equation}
\label{proof-1-3-1}
\left\{
\begin{aligned}
&(\nabla^2 Q+\tilde{\omega}_2) \omega_1>(\nabla^2 f + {\omega}_2) \tilde{\omega}_1,\\
&\kappa_1\leq \kappa \leq \kappa_2, 
\end{aligned}
\right.
\end{equation}
or
\begin{equation}
\label{proof-1-3-2}
\left\{
\begin{aligned}
&(\nabla^2 Q+\tilde{\omega}_2) \omega_1<(\nabla^2 f + {\omega}_2) \tilde{\omega}_1, \\
&\kappa_2\leq \kappa \leq \kappa_1,
\end{aligned}
\right.
\end{equation} 
where
\begin{equation*}
\left\{
\begin{aligned}
&\kappa_1=\frac{ (\nabla^2 f + {\omega}_2)\left[(-\revc{\rho}+1)(\nabla^2 Q+\tilde{\omega}_2)+\tilde{\omega}_1\right]}{(\nabla^2 Q+\tilde{\omega}_2) \omega_1-(\nabla^2 f + {\omega}_2) \tilde{\omega}_1}, \\
&\kappa_2=\frac{ (\nabla^2 f + {\omega}_2)\left[(\revc{\rho}+1)(\nabla^2 Q+\tilde{\omega}_2)+\tilde{\omega}_1\right]}{(\nabla^2 Q+\tilde{\omega}_2) \omega_1-(\nabla^2 f + {\omega}_2) \tilde{\omega}_1}
\end{aligned}
\right.
\end{equation*}
holds. 
}

\end{theorem}

\begin{proof}

Condition (\ref{proof-1-3-1}) or (\ref{proof-1-3-2}) holds if and only if 
\begin{equation*}
\begin{aligned}
\left\|\tilde{\x}_{2}-\x_{2}\right\|_2 & \leq\left\vert 1+\frac{\tilde{\omega}_1(1+\kappa)}{\nabla^2Q+\tilde{\omega}_2}-\frac{\omega_1 \kappa}{\nabla^2 f+{\omega}_2}\right\vert \left\|\tilde{\x}_1-{\x}_{1}\right\|_2 \leq\revc{\rho}\left\|\tilde{\x}_1-{\x}_{1}\right\|_2, 
\end{aligned}
\end{equation*}
based on the basic computation. 
Therefore we prove the conclusion.  
\end{proof}

\begin{coro}
\label{noepsiloncorollary}
Suppose that Assumption \ref{big-assumption-2-k-1} holds, the dimension of the problem \(n=1\)\revc{,} and \(\kappa :=\frac{{\x}_{1}-\x_{0}}{\tilde{\x}_1-{\x}_{1}}\in \revc{\mathbb{R}}\). If \(-1<\kappa<0\), i.e., \(
\tilde{\x}_1\le \x_{0}<{\x}_{1} \text{ or } {\x}_1< \x_{0}\le \tilde{\x}_1,
\) 
then there does not exist an $0<\revc{\rho}<1$ such that (\ref{<q}) holds. 

\end{coro}

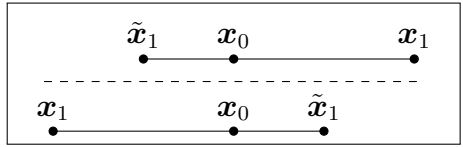
\begin{figure}[H]
\centering
\fbox{
\begin{tikzpicture}[scale=1.2]
% Define coordinates for the points
\coordinate (A) at (0,0);
\coordinate (B) at (1,0);
\coordinate (C) at (3,0);
% Draw the points
\filldraw (A) circle (0.045) node[above]{$\tilde{\boldsymbol{x}}_{1}$};
\filldraw (B) circle (0.045) node[above]{$\boldsymbol{x}_{0}$};
\filldraw (C) circle (0.045) node[above]{$\boldsymbol{x}_{1}$};
% Draw the lines
\draw (A) -- (B);
\draw (B) -- (C);

\draw[dashed] (-1.1,-0.25) -- (3.1,-0.25);

% Define coordinates for the points
\coordinate (D) at (-1,-0.8);
\coordinate (E) at (1,-0.8);
\coordinate (F) at (2,-0.8);
% Draw the points
\filldraw (D) circle (0.045) node[above]{${\boldsymbol{x}}_{1}$};
\filldraw (E) circle (0.045) node[above]{$\boldsymbol{x}_{0}$};
\filldraw (F) circle (0.045) node[above]{$\tilde{\boldsymbol{x}}_{1}$};
% Draw the lines
\draw (D) -- (E);
\draw (E) -- (F);

\end{tikzpicture}
}
\caption{Distribution of \(\x_{0}, {\x}_{1}, \tilde{\x}_1\) corresponding to Corollary \ref{noepsiloncorollary}\label{figure1label}}
\end{figure}

\begin{proof}
Given $\omega_1 > 0, \
\tilde{\omega}_1 > 0, \
G > 0, \
H > 0, \
\revc{\rho}>0, \
-1 \leq \kappa \leq 0$, it holds that 
\begin{align*}
-\revc{\rho} \le 1 + \frac{\tilde{\omega}_1(1+\kappa)}{G} - \frac{\omega_1\kappa}{H}  &\leq \revc{\rho} 
\end{align*}
is equivalent with 
\begin{equation}
\label{-1<kappa<0}
\revc{\rho} \geq \frac{G H-G \kappa  \omega_1+H \kappa  \tilde{\omega}_1+H \tilde{\omega}_1}{G H}\ge 1. 
\end{equation}
Hence the conclusion is proved according to (\ref{-1<kappa<0}). 
\end{proof}

\begin{remark}
\rev{Fig.} \ref{figure1label} \revc{shows the cases in Corollary \ref{noepsiloncorollary}.} 
\end{remark}

\begin{theorem}[sufficient condition for general $n$-dimensional diagonal Hessian case]
\label{sufficient_condition_general_n}
\begin{sloppypar}
{Suppose that Assumption \ref{big-assumption-2-k-1} holds and \rev{\(
\boldsymbol{\kappa}:=\operatorname{diag}\left(\kappa^{[1]}, \kappa^{[2]}, \ldots, \kappa^{[n]}\right)\in\revc{\mathbb{R}}^{n\times n}
\)}  
 satisfies that
\(
\boldsymbol{\kappa}\left(\tilde{\x}_1-{\x}_{1}\right)={\x}_{1}-\x_{0}
\). If for any given \(i\in\{1,2,\dots,n\}\), it holds that 
\begin{equation*} 
\left\{
\begin{aligned} 
&(\nabla^2 Q^{[i]}+\tilde{\omega}_2) \omega_1>(\nabla^2 f^{[i]} + {\omega}_2) \tilde{\omega}_1,\\
&\kappa_1^{[i]}\leq \kappa^{[i]} \leq \kappa_2^{[i]},
\end{aligned}
\right.
\end{equation*}
or
\begin{equation*} 
\left\{
\begin{aligned} 
&(\nabla^2 Q^{[i]}+\tilde{\omega}_2) \omega_1<(\nabla^2 f^{[i]} + {\omega}_2) \tilde{\omega}_1,\\
&\kappa_2^{[i]}\leq \kappa^{[i]} \leq \kappa_1^{[i]},
\end{aligned}
\right.
\end{equation*}
where
\begin{equation*}
\left\{
\begin{aligned}
&\kappa_1^{[i]}=\frac{ (\nabla^2 f^{[i]} + {\omega}_2)\left[(-\revc{\rho}+1)(\nabla^2 Q^{[i]}+\tilde{\omega}_2)+\tilde{\omega}_1\right]}{(\nabla^2 Q^{[i]}+\tilde{\omega}_2) \omega_1-(\nabla^2 f^{[i]} + {\omega}_2) \tilde{\omega}_1},
\\ 
&\kappa_2^{[i]}=\frac{ (\nabla^2 f^{[i]} + {\omega}_2)\left[(\revc{\rho}+1)(\nabla^2 Q^{[i]}+\tilde{\omega}_2)+\tilde{\omega}_1\right]}{(\nabla^2 Q^{[i]}+\tilde{\omega}_2) \omega_1-(\nabla^2 f^{[i]} + {\omega}_2) \tilde{\omega}_1}, 
\end{aligned}
\right.
\end{equation*}
holds, then (\ref{<q}) holds, where the superscript $[i]$ denotes the $i$-th diagonal element of the corresponding matrix \(\nabla^2 f\) or \(\nabla^2 Q\), or the $i$-th element of the correpsonding vectors \(\kappa_1\) and \(\kappa_2\).
}
\end{sloppypar}

\end{theorem}

\begin{proof}

It holds that 
\[
\begin{aligned}
&\Vert \tilde{\x}_{2}-\x_{2}\Vert_2\\
=&\left \Vert \left(\I+\tilde{\omega}_1\left(\nabla^2 Q+\tilde{\omega}_2 \I\right)^{-1}(\I+\boldsymbol{\kappa})-\omega_1\left(\nabla^2 f+{\omega}_2 \I\right)^{-1} \boldsymbol{\kappa}\right) (\tilde{\x}_{1}-\x_{1})\right \Vert_2\\
\le& 
\left \Vert \revc{\rho} (\tilde{\x}_{1}^{[1]}-\x_{1}^{[1]},\dots,\tilde{\x}_{1}^{[n]}-\x_{1}^{[n]})^{\top}  \right \Vert_2= 
\revc{\rho} \Vert \tilde{\x}_{1}-{\x}_{1} \Vert_2,
\end{aligned}
\]
where the superscript $[i]$ denotes the $i$-th element of the corresponding vector, since 
\begin{equation*}
\begin{aligned}
\left\vert 1+\tilde{\omega}_1 \frac{1+\kappa^{[i]}}{\nabla^2 Q^{[i]}+\tilde{\omega}_2}-\omega_1 \frac{\kappa^{[i]}}{\nabla^2 f^{[i]}+{\omega}_2}\right\vert \leq \revc{\rho}, \ \forall\ i=1,\dots,n. 
\end{aligned}
\end{equation*} 
\revc{Then the conclusion is proved based on the above.} 
\end{proof}

\begin{coro}
\label{noepsiloncorollaryndim}
\begin{sloppypar}
Suppose that Assumption \ref{big-assumption-2-k-1} holds, and \rev{\(
\boldsymbol{\kappa}:=\operatorname{diag}\left(\kappa^{[1]}, \kappa^{[2]}, \ldots, \kappa^{[n]}\right)\in\revc{\mathbb{R}}^{n\times n}
\)} 
 satisfies that
\(
\kappa\left(\tilde{\x}_1-{\x}_{1}\right)={\x}_{1}-\x_{0}
\). If \(-1<\kappa^{[i]}<0\) for \(\forall\ i\), 
i.e., 
\(
\tilde{\x}_1^{[i]}\le \x_{0}^{[i]}<{\x}_{1}^{[i]} \text{ or } {\x}_1^{[i]}< \x_{0}^{[i]}\le \tilde{\x}_1^{[i]},
\) 
then there does not exist an $0<\revc{\rho}<1$ such that (\ref{<q}) holds. 
\end{sloppypar}
\end{coro}

\begin{figure}[htbp]
\centering
\fbox{
\begin{tikzpicture}[scale=1]
% Define coordinates for the points
\coordinate (A) at (0,-1);
\coordinate (B) at (1,0);
\coordinate (C) at (3,2);
% Draw the points
\filldraw (A) circle (0.0695) node[above]{$\tilde{\boldsymbol{x}}_{1}$};
\filldraw (B) circle (0.0695) node[above]{$\boldsymbol{x}_{0}$};
\filldraw (C) circle (0.0695) node[above]{$\boldsymbol{x}_{1}$};
% Draw the lines
\draw (A) -- (B);
\draw (B) -- (C);

\draw[dashed] (3.75,2.5) -- (3.75,-1.7);

% Define coordinates for the points
\coordinate (D) at (4.5,-2);
\coordinate (E) at (6.5,-0);
\coordinate (F) at (7.5,1);
% Draw the points
\filldraw (D) circle (0.0695) node[above]{${\boldsymbol{x}}_{1}$};
\filldraw (E) circle (0.0695) node[above]{$\boldsymbol{x}_{0}$};
\filldraw (F) circle (0.0695) node[above]{$\tilde{\boldsymbol{x}}_{1}$};
% Draw the lines
\draw (D) -- (E);
\draw (E) -- (F);

\end{tikzpicture}
}
\caption{Distribution of \(\x_{0}, {\x}_{1}, \tilde{\x}_1\) corresponding to Corollary \ref{noepsiloncorollaryndim}\label{figure2label}}
\end{figure}
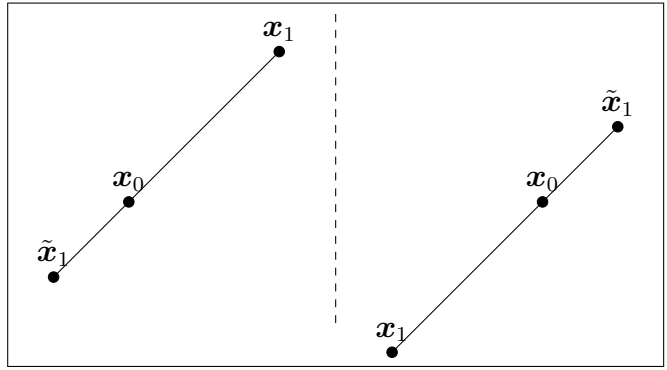

\begin{proof}
The conclusion can be directly derived based on the conclusion of each element of \(\tilde{\x}_{2}-\x_{2}\) from the proof of Corollary \ref{noepsiloncorollary}. 
\end{proof}

\begin{remark}%\color{red}
\rev{Fig.} \ref{figure2label} shows the cases in Corollary \ref{noepsiloncorollaryndim}. 
\end{remark}

\section{\rev{Numerical examples}}

\rev{We present the following examples to illustrate our results.}  %the conclusions above. 

\begin{example}
\revc{In this example, we show the case where dimension $n=2$, \rev{local quadratic approximate function}  has diagonal Hessian, and $\boldsymbol{\kappa}$ has different non-zero components.} 
\[
\begin{aligned}
&\left\{\begin{aligned}
&f(\x)=-\frac{1}{2} \x^{\top}
\begin{pmatrix}
1 & 0 \\ 
0 & 2
\end{pmatrix}
 \x+\left( 
\frac{1}{7}, \frac{5}{3}
\right) \x, \\
 &Q(\x)=-\frac{1}{2} \x^{\top}
 \begin{pmatrix}
 1 & 0 \\
  0 & 1
  \end{pmatrix}
 \x.
  \end{aligned}\right.
 \end{aligned}
  \]
Besides, 
\(
\x_0=(1, 1)^{\top},\ \omega_1=3,\ \tilde{\omega}_1=3,\ {\omega}_2=4,\ \tilde{\omega}_2=5,\ \revc{\rho}=\frac{1}{2 }.
\) 
We have 
\begin{equation*}
% \begin{aligned}
\x_1=\x_0-\left(\nabla^2 f+\omega_1 \I\right)^{-1} \nabla f 
=\begin{pmatrix}\frac{10}{7} \\ \frac{4}{3}\end{pmatrix},\ 
\tilde{\x}_1=\x_0-\left(\nabla^2 Q+\tilde{\omega}_1 \I\right)^{-1} \nabla Q 
=\begin{pmatrix}\frac{3}{2} \\ \frac{3}{2}\end{pmatrix},
% \end{aligned}
\end{equation*}
and  
\[
\boldsymbol{\kappa}=\left(\begin{array}{cc}\frac{\frac{10}{7}-1}{\frac{3}{2}-\frac{10}{7}} & 0 \\ 0 & \frac{\frac{4}{3}-1}{\frac{3}{2}-\frac{4}{3}}\end{array}\right)=\left(\begin{array}{cc}6 & 0 \\ 0 & 2\end{array}\right). 
\]

It holds that
\begin{equation*}
\left\{
\begin{aligned} 
&(\nabla^2 Q^{[1]}+\tilde{\omega}_2) \omega_1=12>9=(\nabla^2 f^{[1]} + {\omega}_2) \tilde{\omega}_1,\\
&\kappa_1^{[1]}\leq \kappa^{[1]} \leq \kappa_2^{[1]},
\end{aligned}
\right.
\end{equation*}
and 
\begin{equation*} 
\left\{
\begin{aligned} 
&(\nabla^2 Q^{[2]}+\tilde{\omega}_2) \omega_1=12>6=(\nabla^2 f^{[2]} + {\omega}_2) \tilde{\omega}_1,\\
&\kappa_1^{[2]}\leq \kappa^{[2]} \leq \kappa_2^{[2]},
\end{aligned}
\right.
\end{equation*} 
where
\begin{equation*}
\left\{
\begin{aligned}
&\kappa_1^{[1]}=\frac{ (\nabla^2 f^{[1]} + {\omega}_2)\left[(-\revc{\rho}+1)(\nabla^2 Q^{[1]}+\tilde{\omega}_2)+\tilde{\omega}_1\right]}{(\nabla^2 Q^{[1]}+\tilde{\omega}_2) \omega_1-(\nabla^2 f^{[1]} + {\omega}_2) \tilde{\omega}_1}=-4\revc{\rho}+7=5, \\
&\kappa_2^{[1]}=\frac{ (\nabla^2 f^{[1]} + {\omega}_2)\left[(\revc{\rho}+1)(\nabla^2 Q^{[1]}+\tilde{\omega}_2)+\tilde{\omega}_1\right]}{(\nabla^2 Q^{[1]}+\tilde{\omega}_2) \omega_1-(\nabla^2 f^{[1]} + {\omega}_2) \tilde{\omega}_1}=4\revc{\rho}+7=9,
\end{aligned}
\right.
\end{equation*}
and 
\begin{equation*}
\left\{
\begin{aligned}
&\kappa_1^{[2]}=\frac{ (\nabla^2 f^{[2]} + {\omega}_2)\left[(-\revc{\rho}+1)(\nabla^2 Q^{[2]}+\tilde{\omega}_2)+\tilde{\omega}_1\right]}{(\nabla^2 Q^{[2]}+\tilde{\omega}_2) \omega_1-(\nabla^2 f ^{[2]}+ {\omega}_2) \tilde{\omega}_1}=-\frac{4}{3}\revc{\rho}+\frac{7}{3}=\frac{5}{3}, \\
&\kappa_2^{[2]}=\frac{ (\nabla^2 f^{[2]} + {\omega}_2)\left[(\revc{\rho}+1)(\nabla^2 Q^{[2]}+\tilde{\omega}_2)+\tilde{\omega}_1\right]}{(\nabla^2 Q^{[2]}+\tilde{\omega}_2) \omega_1-(\nabla^2 f^{[2]} + {\omega}_2) \tilde{\omega}_1}=\frac{4}{3}\revc{\rho}+\frac{7}{3}=3,
\end{aligned}
\right.
\end{equation*}
and thus it satisfies the sufficient condition. Besides, we obtain that  
\begin{align*}
&\tilde{\x}_{2}-\x_{2} \\
=&\tilde{\omega}_1\left(\nabla^{2} Q+\tilde{\omega}_2 \boldsymbol{I}\right)^{-1}  (\tilde{\x}_{1}-\x_{0})-\omega_1\left(\nabla^{2} f+{\omega}_2 \boldsymbol{I}\right)^{-1} (\x_{1}-\x_{0})+(\tilde{\x}_1-\x_1)\\
=&\begin{pmatrix}
\frac{1}{56} \\
\frac{1}{24}
\end{pmatrix},
\end{align*} 
and then 
\[
\Vert \tilde{\x}_2-\x_2\Vert_2=\frac{\sqrt{\frac{29}{2}}}{84}<\frac{1}{2}\frac{\sqrt{\frac{29}{2}}}{21}=\revc{\rho} \Vert \tilde{\x}_1-\x_1\Vert_2.
\]

\end{example}

\rev{The following example shows a numerical observation of the coefficients making the error distance reduced in the 1-dimensional case.} 
%making the distance reduce in the case where the dimension \(n=1\).}

\begin{example}
\label{example-3.2}
\rev{We figure out the probability that the coefficients satisfy the conditions in Theorem \ref{thm-1-1} and illustrate the 1-dimensional case.} We perform a numerical experiment using the software Mathematica \rev{(Version 13.3)}\footnote{Codes are available in \href{https://github.com/PengchengXieLSEC/distance-reduction}{https://github.com/PengchengXieLSEC/distance-reduction}.} to compute the integral \rev{of a boolean expression giving the probability measure for different values} of $q$ and $\revc{\rho}$. 
Specifically, we compute the integral over ${\omega}_2$ and $\tilde{\omega}_2$ separately in the ranges \([0,q\omega_1]\) and \([0,q\tilde{\omega}_1]\), where \(q\) is a non-negative real coefficient. We then divided the result by $q^2 \omega_1 \tilde{\omega}_1$ to represent the probability, i.e., 
\begin{equation*}
\begin{aligned}
&\text{Prob}(\revc{\rho})\\
=&\frac{1}{q^2 {\omega_1} {\tilde{\omega}_1}}\int_0^{q {\omega_1}}\int _0^{q {\tilde{\omega}_1}}\text{Boole}\left[\nabla^2 Q+{\tilde{\omega}_2}\geq -\frac{(\kappa +1) {\tilde{\omega}_1} (\nabla^2 f+{{\omega}_2})}{(\revc{\rho}+1)  {({\omega}_2+\nabla^2 f)}-\kappa  {\omega_1}}\right] \\
&\text{Boole}\left[ \nabla^2 Q+{\tilde{\omega}_2}\leq \frac{(\kappa +1) {\tilde{\omega}_1} (\nabla^2 f+{{\omega}_2})}{(\revc{\rho} -1) {({\omega}_2+\nabla^2 f)}+\kappa  {\omega_1}}\right]d{\tilde{\omega}_2}d{{\omega}_2},
\end{aligned} 
\end{equation*}
where \(\text{Boole}(\cdot)\) denotes the 0/1-output Boolean function. 
Notice that in this example, we define the constants $\nabla^2 Q = -1$, $\nabla^2 f = -2$, $\omega_1 = 3$, $\tilde{\omega}_1 = 3$, $\kappa = -2$, and  
\(
q= 10^{-3}, 10^{-2}, 10^{-1}, 1, 10, 10^2, 10^3.
\)

\begin{figure}[htb]
  \centering
  \includegraphics[width=0.95\textwidth]{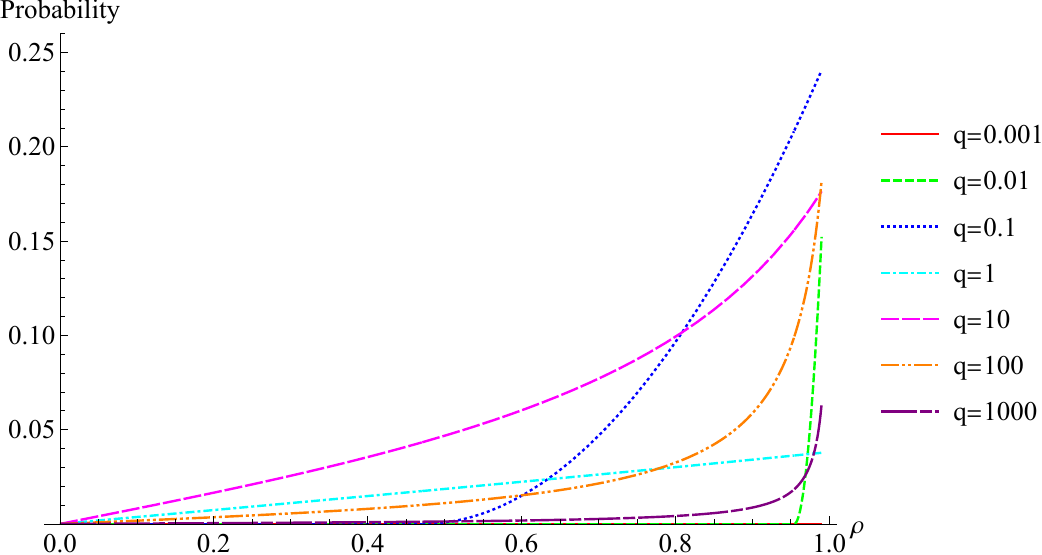}
%  \caption\revc{Numerical results of Example 2}
\caption{\rev{Probability of error distance reduction in Example \ref{example-3.2}}}
  \label{fig:myPlot}
\end{figure}
\end{example}
%Numerical results for the function $\text{Prob}(\revc{\rho})$ as a function of the perturbation parameter $\revc{\rho}$. The different lines correspond to different values of the parameter $m$.

\revc{\rev{Fig.} \ref{fig:myPlot} shows the numerical results for the function $\text{Prob}(\revc{\rho})$ as a function of the perturbation parameter $\revc{\rho}$. The different lines correspond to different values of the parameter $q$.}

From {\rev{Fig.}} \ref{fig:myPlot}, it can be seen that, in this 1-dimensional example, the probability of obtaining the coefficients ${\omega}_2$ and $\tilde{\omega}_2$ achieving the distance reduction is at most about 25\%, where the corresponding integrate region is \([0, 10^{-1} \omega_1]\times [0, 10^{-1} \tilde{\omega}_1]\).

\section{Conclusions and discussions}

\rev{This paper analyzes} sufficient conditions for the \rev{reduction in distance} between the minimizers of {\rev{nonconvex}}  quadratic functions in the trust region after two iterations. \rev{Note that quadratic functions are frequently used for local approximation in numerical optimization algorithms, but obtaining an accurate approximation is often challenging in most nonlinear cases.}  %Notice that the quadratic functions are often used for the local approximation by the numerical optimization algorithms, and we could not obtain an accurate one in most nonlinear cases. 
If we have different \rev{local quadratic approximate function}s, the results in this paper provide a method to reduce the distance of the minimizers of different \rev{local quadratic approximate function}s by selecting the diagonal damping coefficients of the \rev{approximate function}s' Hessian \({\omega}_2\) and \(\tilde{\omega}_2\) accordingly. Besides, the example above shows that in some cases, the \rev{error distance} of the minimizers of the \rev{local quadratic approximate function} will increase with a high probability after one more iteration, and this derives the necessity of modifying the \rev{local quadratic approximate function} at each step in the optimization methods based on quadratic \rev{approximate function}s. In other words, one \rev{local quadratic approximate function} used by the trust-region methods is supposed to be updated after only one iteration, even if the \rev{approximate function} is {\rev{nonconvex}} and the iteration step reaches the bound of the trust region.

\noindent {\bf Acknowledgments}

The authors would like to thank the editors and the anonymous referees for their careful reading and providing valuable suggestions.

%The author  gratefully acknowledges the great help and support of  Professor Ya-xiang Yuan. The work was partially supported by National Natural Science Foundation of China (No. 11688101 and 12288201).
%

%%%~\\

\noindent {\bf Data availability statements}

The codes that support the findings of this study are openly available in \href{https://github.com/PengchengXieLSEC/distance-reduction}{https://github.com/PengchengXieLSEC/distance-reduction}.

\section*{Declarations}

The author has no relevant financial or non-financial interests to disclose.

% \section{}
% \label{}



 % \bibliographystyle{elsarticle-num} 
 % \bibliography{sn-bibliography}


\end{document}